\documentclass[sn-mathphys-num]{sn-jnl}% Math and Physical Sciences Numbered Reference Style 
\usepackage{graphicx}%
\usepackage{multirow}%
\usepackage{amsmath,amssymb,amsfonts}%
\usepackage{amsthm}%
\usepackage{mathrsfs}%
\usepackage[title]{appendix}%
\usepackage{xcolor}%
\usepackage{textcomp}%
\usepackage{manyfoot}%
\usepackage{booktabs}%
\usepackage{algorithm}%
\usepackage{algorithmicx}%
\usepackage{algpseudocode}%
\usepackage{listings}%
\newtheorem{theorem}{Theorem}%  meant for continuous numbers
\newtheorem{proposition}[theorem]{Proposition}% 
\newtheorem{corollary}[theorem]{Corollary}
\newtheorem{observation}[theorem]{Observation}
\newtheorem{lemma}[theorem]{Lemma}

\newtheorem{conjecture}[theorem]{Conjecture}

 \newtheorem*{definition*}{Definition}
\newtheorem{problem}[theorem]{Open Problem}

\title{The minimum number of peeling sequences of a point set} %TODO Please add

\begin{document}

\author*[1,2,3]{\fnm{Dániel G.} \sur{Simon}}\email{dgsimon@renyi.hu}

\affil[1]{\orgdiv{Alfréd Rényi Institute of Mathematics, Hungary}}

\affil[2]{\orgdiv{Eötvös Loránd University, Hungary}}

\affil[3]{Research supported by ERC advanced grant no. 882971: Geoscape}

 %TODO mandatory; please add comma-separated list of keywords

%\supplement{}%optional, e.g. related research data, source code, ... hosted on a repository like zenodo, figshare, GitHub, ...
%\supplementdetails[linktext={opt. text shown instead of the URL}, cite=DBLP:books/mk/GrayR93, subcategory={Description, Subcategory}, swhid={Software Heritage Identifier}]{General Classification (e.g. Software, Dataset, Model, ...)}{URL to related version} %linktext, cite, and subcategory are optional

%\funding{(Optional) general funding statement \dots}%optional, to capture a funding statement, which applies to all authors. Please enter author specific funding statements as fifth argument of the \author macro.

%\nolinenumbers %uncomment to disable line numbering

%Editor-only macros:: begin (do not touch as author)%%%%%%%%%%%%%%%%%%%%%%%%%%%%%%%%%%
%\EventEditors{}
%\EventNoEds{0}
%\EventLongTitle{The 40th International Symposium on Computational Geometry}
%\EventShortTitle{SoCG 2024}
%\EventAcronym{SoCG}
%\EventYear{2024}
%\EventDate{June 11--14, 2024}
%\EventLocation{Athens, Greece}
%\EventLogo{}
%\SeriesVolume{}
%\ArticleNo{282}
%%%%%%%%%%%%%%%%%%%%%%%%%%%%%%%%%%%%%%%%%%%%%%%%%%%%%%

\def\R{\mathbb{R}}
\def\p{\mathbf{p}}

\abstract{\linespread{1}\selectfont
 \noindent Let $P$ be a set of $n$ points in $\R^d$, in general position. We remove all of them one by one, in each step erasing one vertex of the convex hull of the current remaining set. Let $g_d(P)$ denote the number of different removal orders we can attain while erasing all points of $P$ this way, and let $g_d(n)$ be the \emph{minimum} of $g_d(P)$ over all $n$-element point sets $P\subset \R^d$. Dumitrescu and Tóth showed that $g_d(n)\le(d+1)^{(d+1)^2n}$. We substantially improve their bound, by proving that $g_d(n)=O((d+d\ln{d})^{(2+\frac{(d-1)}{\lfloor d\ln{d}\rfloor})n})$. It follows that, for any $\epsilon>0$, there exist sufficiently high dimensional point sets $P\subset \R^d$ with $g_d(P)\leq O(d^{(2+\epsilon)n})$. This almost closes the gap between the upper bound and the best-known lower bound $(d+1)^n$ for large values of $d$.}

  
\maketitle

\footnotetext{No datasets were generated or analysed during the current study}

\section{Introduction}

Let $P$ be a set of $n>d$ points labelled ${1,2,...,n}$ in $\R^d$ in general position, i.e., assume that no $d+1$ of them lie on the same hyperplane. In each step, we remove exactly one vertex of the convex hull of the current point set and write down the label of the removed point. We repeat the process until removing all the points in $P$. We get a sequence of the labels, and call that sequence a "peeling sequence of $P$". We are interested in determining the \emph{minimum} number of peeling sequences a set of $n$ points can have. The maximal value is not that interesting, as we have $g_d(P)\le n!$ for every set of $n$ points, and $g_d(P)=n!$ for every set of $n$ points in convex position.

\begin{definition*}
Given a point set $P$ in general position in $\R^d$, let $g_d(P)$ denote the number of peeling sequences of $P$. 

Let 
$g_d(n)=\min_{|P|=n} g_d(P)$, where the minimum is taken over all $n$-element point sets $P$ in general position in $\R^d$.
\end{definition*}

The investigation of peeling sequences was initiated by Dumitrescu~\cite{dum1}, who studied only the $2$-dimensional problem. For $g_2(n)$, the minimum number of peeling sequences, he established the following asymptotic bounds:
$$\Omega(3^n)\le g_2(n)\leq 2^{O(n\log{\log{n}})}.$$

The upper bound later got improved significantly by Dumitrescu and G. Tóth \cite{Dumitrescu}, where they have shown $g_2(n)\leq \frac{12.29^n}{100}$. 
 
 The lower bound is simple, and can be generalized to any dimension $d$: In each step of the peeling process, where we have more than $d$ points left, the current convex hull must consist of at least $d+1$ points. Hence, in each step, we have at least $d+1$ choices to continue the peeling sequence, which gives at least $\Omega((d+1)^n)$ different sequences on $n$ points. This argument gives the best known lower bound on $g_d(n)$.
 
 In the other direction, the best known upper bounds are also due to Dumitrescu and G.~Tóth \cite{Dumitrescu}:  $$g_d(n)\leq (d+1)^{(d+1)^2n}\;\; \mbox{for every}\;\; d$$.

In the current paper, we substantially improve the above bound, for $d\geq 3$. Our main result is the following. 

 \begin{theorem}\label{1}
 For any fixed $d\geq3$ and for all $n\geq2$ integers, there is a set of $n$ points in general position in $\R^d$, that admits at most $O((d+d\ln{d})^{(2+\frac{d-1}{\lfloor d\ln{d}\rfloor})n})$ peeling sequences. In other words, we have
 $$g_d(n)\leq O((d+d\ln{d})^{(2+\frac{d-1}{\lfloor d\ln{d}\rfloor})n}).$$
\end{theorem}

Theorem~\ref{1} immediately follows from the next statement, by choosing $m=d\ln{d}$.

\begin{theorem}\label{0}
     For any fixed integers $d\ge 3,\, m\geq1,$ and all integers $n\ge 2$, we have $$g_d(n)\leq c\cdot(d+m)^{\frac{d+2m-1}{m}n},$$ where $c$ is an absolute constant.
\end{theorem}

As $d$ tends to infinity, the upper bound in Theorem \ref{1} is asymptotically not far from $d^{2n}$.

\begin{corollary} \label{2}
    For any $\epsilon>0$, there is a real number $D=D(\epsilon)$ such that $g_d(n)=O(d^{(2+\epsilon)n})$, for all fixed $d\ge D$.
\end{corollary}

   We build our point sets recursively: we place in $\R^d$, in a regular manner, tiny copies of our construction for smaller values of $n$, transformed into segment-like objects.

   To find the optimal placement of these segments, we need to define and solve a different problem. Suppose the origin $\mathbf{p}$ belongs to our point set, and we want to make sure that $\mathbf{p}$ does not get removed in the first $m$ steps of the peeling algorithm. We seek the minimal number of points needed to achieve that. These numbers directly affect the number of peeling sequences.
\smallskip

   \textbf{Related work.} The number of peeling sequences is an important parameter of point sets, which can be used for their classification. This number is maximal when the points are in convex position. To some extent, this value measures how far the point set is from being in convex position. Similar parameters that have been investigated are the number of \emph{triangulations} (see \cite{triang3, triang2, triang1} for some bounds), the number of \emph{polygonizations} (see \cite{poly1}, \cite{poly2}), the number of \emph{non-crossing perfect matchings} (see \cite{matching1}), etc.

    Peeling point sets has also been widely studied in several other contexts. Many results are concerned with \emph{convex-layer peeling}, where at each step we remove all vertices of the current convex hull. Such methods have been used, e.g., for proving that every sufficiently large set of points in general position in the plane has six points that form an empty hexagon \cite{gerken}\cite{nicolas}; see also \cite{dumremark}. Another area where peeling the \emph{grid} plays an important role is \emph{affine curve-shortening} \cite{alvarez, grid2, sapiro}, which has applications in computer graphics. The number of layers a point set has is called the \emph{layer number}. Bounds on the layer number of various structures have been investigated, see \cite{even1, even2, grid1} for some examples.
   
     In Section $2$, we define and analyse the auxiliary problem needed for our construction, which is interesting on its own right. In Section $3$, we describe our recursive construction and prove Theorem \ref{0}. In Section $4$, we prove our numerical theorems, Theorem \ref{1} and Corollary \ref{2}. The last section contains concluding remarks and open problems. 

   \section{Defending a point}

To define the constructions needed for the upper bounds in Theorem \ref{0}, we need to solve a completely different problem first.

 The new problem is the following: For a fixed dimension $d$, let $\p$ be the origin in $\R^d$. Find a point set $S$, such that $S\cup \{\mathbf{p}\}$ is in general position, and if we start peeling $S\cup \{\mathbf{p}\}$, $\p$ cannot get removed in the first $m$ steps. We seek to minimise $|S|$. We say that the set $S$ \emph{defends} $\p$ for $m$ steps.

\begin{definition*}[Defense number] \label{defense}
    For given positive integers $d,m$, let $D_d(m)$ denote the size of the minimal set $S$ in $\R^d$ needed to defend the origin $\p$ for the first $m$ steps of the peeling algorithm.
\end{definition*}

  \begin{observation}\label{halfobs}
    A point in $\R^d$ is on the convex hull of a point set in general position, if and only if there is a hyperplane through the point, such that there are no points in one of the open halfspaces bounded by the hyperplane.
    \end{observation}

 Based on the above observation, we can give an equivalent definition of the Defense number:

\begin{definition*}[Defense number] \label{defense2}
    For given positive integers $d,m$, $D_d(m)$ denotes the size of the minimal set $S$ in $\R^d$ needed to have at least $m$ points in any of the open halfspaces around the origin.
\end{definition*}

 Note that there is a similar notion in previous papers called Tukey-depth \cite{Tukey}, but it is defined via closed halfspaces, so it is not exactly equivalent to our Defense number.

 Now we determine the exact values of $D_d(m)$ for all $d$ and $m$. As a warm-up, we exhibit two simple cases. In $1$-dimension, general position means that no two points lie in the same position.

\begin{proposition} \label{easyv}
The following values are easily seen:

 a) $D_d(1)=d+1$.
 
 b) $D_1(m)=2m$.

\end{proposition}

\begin{proof} $\\$
    \begin{itemize} 
    
        \item[a)] By the definition of defending for $1$ step, we need a set $S$ such that $\p$ lies in the interior of $conv(S)$. The convex hull of at most $d$ points must be degenerate in $\R^d$, so $\p$ cannot be in its interior. Hence we need $|S|\geq d+1$. To achieve that, we can select a regular simplex centered at the origin, as illustrated in Figure \ref{fig:regular simplex}.
        \item[b)] Since the dimension is $1$, points are arranged on a single line. If we want to defend the origin in the first $m$ steps, we need to place $m$ points to the left of $\p$ on the line, and $m$ points to the right, and it is obviously the optimal construction. See Figure \ref{fig:1d}.
    \end{itemize}
\end{proof}
\begin{figure}
	\begin{center}
		\includegraphics[scale=0.4]{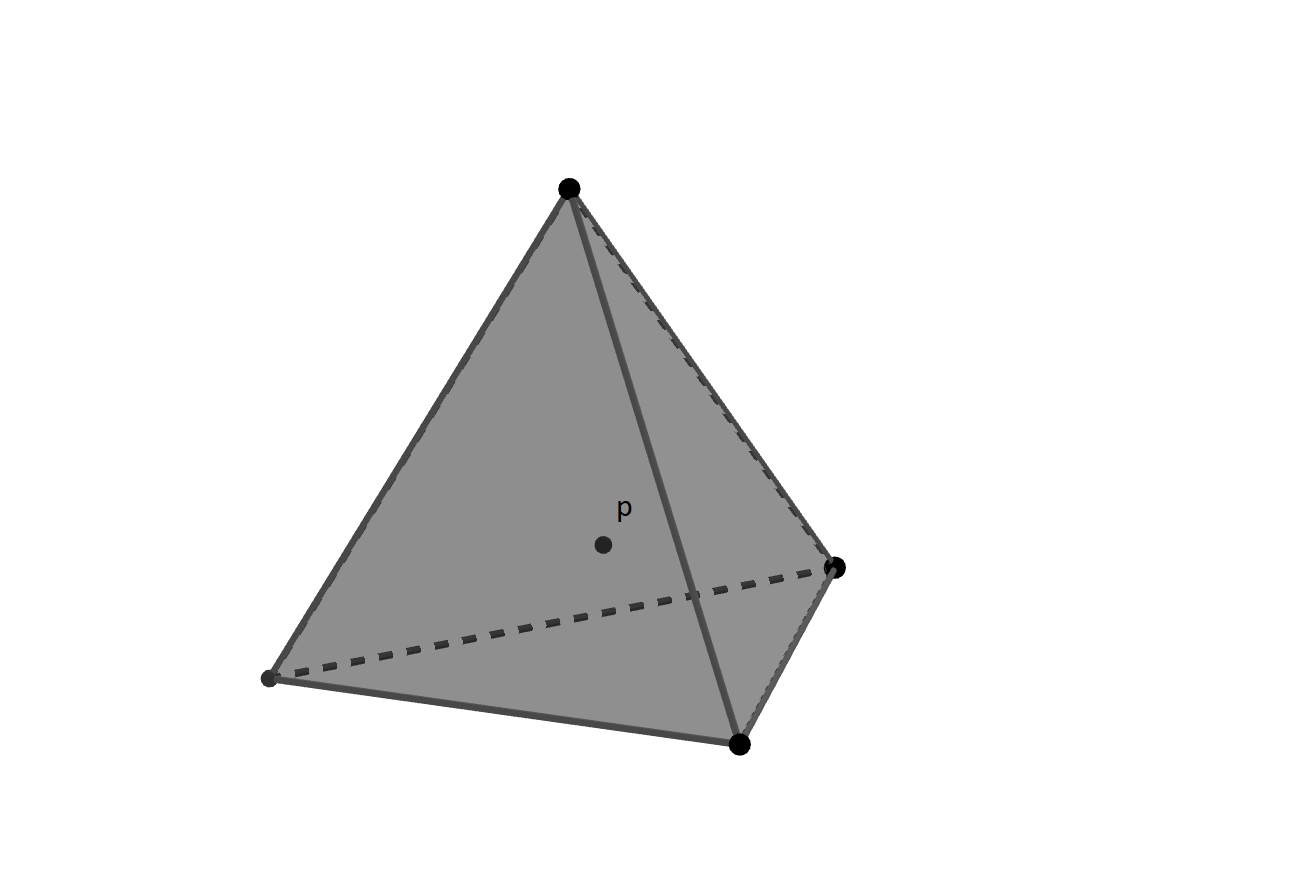}
	
		\label{fig:regular simplex}
	\end{center}
\end{figure}
\begin{figure}
	\begin{center}
		\includegraphics[scale=0.7]{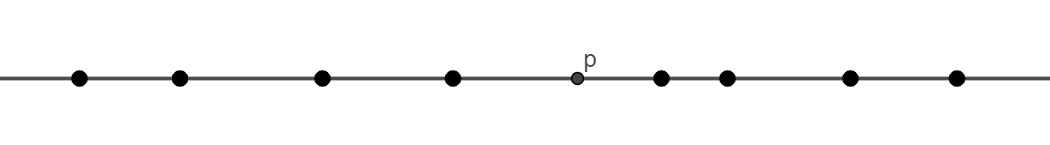}
		\caption{A minimal point set for $d=3$ and $m=1$ defending $\p$ (upper) \\ A minimal point set for $d=1$ and $m=4$ defending $\p$ (lower)}
		\label{fig:1d}
	\end{center}
\end{figure}

 Now we state and prove the general value of $D_d(m)$, which is the main result of this section. Theorem \ref{maindefense} will be used for building up constructions, and proving the main theorems of the paper.

\begin{theorem}\label{maindefense}
    For every $m\geq1$ integer, the minimal number of points needed to defend point $\p$ for $m$ steps in $\R^d$ is $D_d(m)=d+2m-1$.
\end{theorem}

    \begin{proof} [Proof of upper bound]
    The upper bound is established through Gale's Theorem II., first presented in \cite{Gale}, and later simplified in \cite{Petty}. The theorem says, that one can place $2m+d-1$ points on $S_{d}$ (the surface of the $d+1$ dimensional sphere), such that each open hemisphere through the origin contains at least $m$ points.
    \end{proof}
    
    \begin{proof}[Proof of lower bound]
    We prove the lower bound for all point sets, not just the ones in general position. This includes having more than one point at the same position too. 
    
    By Proposition \ref{easyv} b), in one dimension, the lower bound holds, and it also remains valid without general position. For higher dimensions, we proceed by induction on $d$.
    Suppose that $D_d(m)\geq d+2m-1$ holds for any $m$ and $d\leq d_0$, where $m, d, d_0$ are all positive integers. We aim to show that the lower bound holds for $d=d_0+1$ too.

    Suppose for contradiction, that there is an $m$, such that the origin $\p$ can be defended for $m$ steps in $\R^{d_0+1}$ by $d_0+2m-1$ points. Let $S$ be such a defending point set, and let $\mathbf{v}\in S$. We can also suppose $\mathbf{v}$ is not the origin. Let $\mathbf{H}$ be the $d_0$-dimensional hyperplane through the origin, whose normal vector is $\vec{\mathbf{v}\p}$. Project all elements of $S\setminus\{\mathbf{v}\}$ onto $\mathbf{H}$. 

    After projection, we get $d_0+2m-2$ points around $\p$ in $\R^d$, not necessarily in general position. By the induction hypothesis, the origin cannot be defended for $m$ steps with that many points, hence there is a peeling sequence starting with $\mathbf{v_{i_1}',v_{i_2}',...v_{i_{m-1}}',p}$ on $\mathbf{H}$, where $\mathbf{v_i'}$ denotes the image of $\mathbf{v_i}$ after the projection. So let $\mathbf{v_{i_1}',v_{i_2}',...v_{i_{m-1}}',p}$ be the beginning of the peeling sequence of the image on $\mathbf{H}$.
    We claim $\mathbf{v_{i_1},v_{i_2},...,v_{i_{m-1}},p}$ is the beginning of a valid peeling sequence in $\R^{d_0+1}$, contradicting that $S$ defends $\p$ for $m$ steps.

    If we forget the existence of $\mathbf{v}$, the correctness of the peeling sequence is obvious, since there is a basis of vectors in which the first $d_0$ coordinates of the points are unchanged by reversing the projections. Therefore if at a given step, point $\mathbf{v_i'}$ is on the convex hull in the hyperplane, it is not a linear combination with non-negative weights of sum $<1$ of the other remaining points. This condition also stays true by taking preimages, since the first $d_0$ coordinates are unchanged. So the preimage of the peeling sequence is a peeling sequence in $\R^d$ too.

    Adding back $\mathbf{v}$ does not affect that, since the image of $\mathbf{v}$ would be $\p$, so they reach the boundary at the same time in $\mathbf{H}$, and by taking preimage, the line $\mathbf{vp}$ becomes a boundary segment of the current convex hull. So whenever $\mathbf{v}$ gets into the peeling sequence in the above way, we can switch it to $\p$.

    Hence we have a contradiction, thus we need at least $d_0+2m$ points to defend the origin for $m$ steps in $\R^{d_0+1}$, satisfying the induction step and concluding the proof.
    \end{proof}

    \section{Our Construction}

    \subsection{Constructing Base Sets}
        The construction is built recursively. The main idea is to take the optimal constructions for smaller values of $n$, transform them to look like a tiny segment, and then place copies of them in $\R^d$ such that their placement points defend the origin for many steps. We will call the set of these placement points $S$, which we soon define. This way, if the segments point towards the origin, as long as the origin is defended, we can only peel the outermost point of each segment. Consider a fixed dimension $d$, and a fixed positive integer $m$, and these stay fixed during Section 3. For each positive integer $n$ we will define a construction $S_n$, using an optimal $m$-defending set $S'$ of size $D=D_d(m)=d+2m-1$.

        Let $S'=\{\mathbf{p'_1,p'_2,...p'_{D}}\}$ be a set of minimal size in $\R^d$ defending the origin for $m$ steps in the peeling algorithm. From $S'$, we derive a special set $S$ of the same size as $S'$, in which during the first $m$ steps of the peeling algorithm, there are never more than $D-m+1=d+m$ points on the convex hull. We refer to $S$ as the base set of the construction.

  \begin{figure}
	\begin{center}
		\includegraphics[scale=0.42]{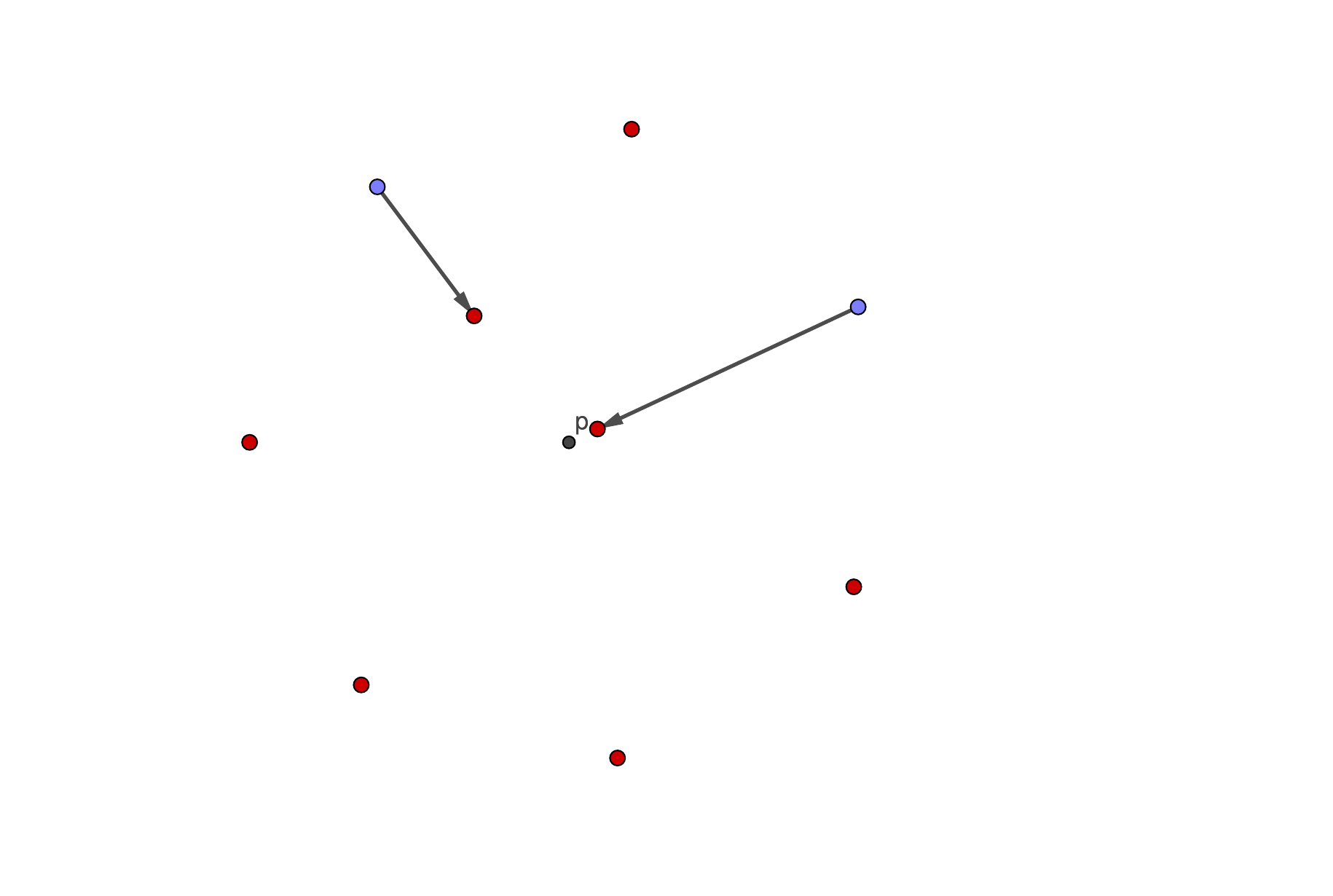}
		\caption{Creation of the base set starting from a $3$-defending point set in the plane. Red points show the points in the final set.}
		\label{fig:base}
	\end{center}
\end{figure}       

  Such a defending set can be derived from an arbitrary $S'=\{\mathbf{p_1', p_2',..., p_{D_d(m)}'}\}$ $m$-step defending set of size $D$. Rescale $S'$ such that each point has a distance $\leq1$ from the origin. Let $\mathbf{p_j}=\mathbf{p_j'}$ for all $j\leq D-m+1$. Note that $\{\mathbf{p_1,p_2,..., p_{D-m+1}}\}$ defends the origin $\p$ for $1$ step, so there is a closed ball of center $\p$ and radius $\epsilon_1$ contained in the interior of the convex hull of the set. Let $\mathbf{p_{D-m+2}}=\epsilon_1\mathbf{p_{D-m+2}'}$, where multiplying a point by a scalar $\epsilon$ denotes the transformation $(x_1, x_2,..., x_d)\rightarrow(\epsilon x_1, \epsilon x_2,..., \epsilon x_d)$.
 By Observation \ref{halfobs}, it is easy to see that scaling distances of individual points from the origin does not change the defense properties of a point set.
 So, $\{\mathbf{p_1,p_2,..., p_{D-m+2}}\}$ defends the origin $\p$ for $2$ steps, so there is a closed ball of radius $\epsilon_2$ centered at $\p$, contained in the interior of the convex hull of any subset of $D-m+1$ different points among these $D_d(m)-m+2$ points. Let $\mathbf{p_{D-m+3}}=\epsilon_2 \mathbf{p_{D-m+3}'}$. Similarly, recursively, for any $j\leq m-1$, define $\epsilon_j$ the same way, and define $\mathbf{p_{D-m+j+1}}=\epsilon_j\mathbf{p'_{D-m+j+1}}$. The set $S=\{\mathbf{p_1, p_2,..., p_{D}}\}$ satisfies the properties described at the beginning of the subsection.

  We can notice that in the resulting set, $\mathbf{p_{D-m+j}}$ can only be removed at or after the $j$-th step for all $j$, hence in each step, only one more point is added to the new convex hull. Initially, there are at most $D-m+1$ points on the convex hull, so it stays like that in all of the steps. The creation of our base set from an arbitrary set is illustrated in Figure \ref{fig:base} for $d=2$ and $m=3$.

\begin{figure}
	\begin{center}
		\includegraphics[scale=0.41]{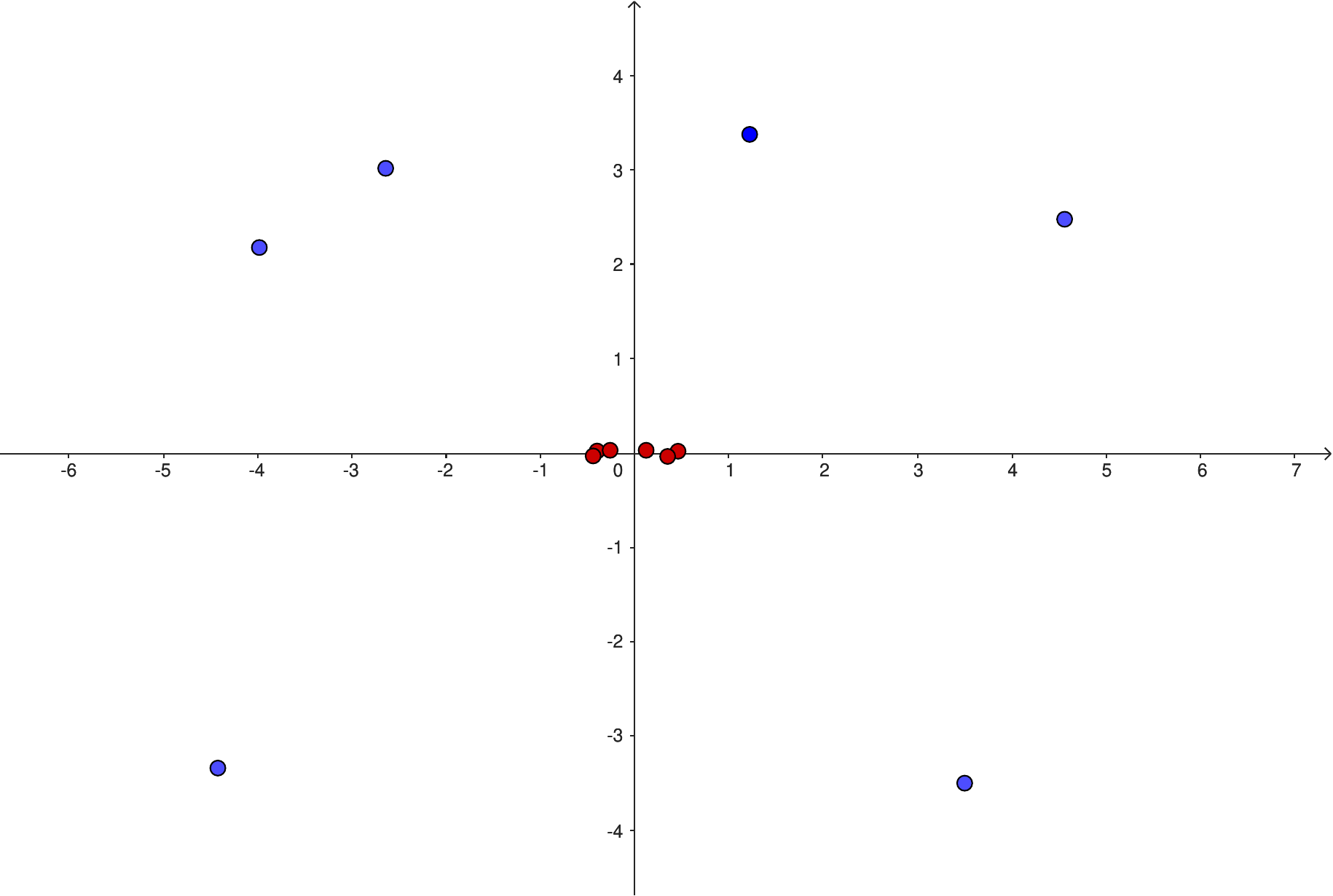}
		\caption{Transforming a point set on the plane with $\delta=0.1$ and $\epsilon=0.01$. The blue points are the original points, and the red ones are the transformed points. In practice, much smaller $\delta$ and $\epsilon$ values will be used.}
		\label{fig:segment}
	\end{center}
\end{figure}

  \subsection{Constructing $S_n$}

        The point sets of the constructions are built recursively. Let $S_n$ denote the construction for $n$ points. For $n\leq D$, pick $S_n$ to be an arbitrary subset of $S$.

       For $n>D$, partition $n$ into $D$ pieces of size $i_1,i_2,...,i_D$, such that all the pieces have size $\lfloor \frac nD\rfloor$ or $\lceil \frac nD \rceil$. Then, for any $i_j$, take a copy of $S_{i_j}$, and transform it to look like a tiny segment. This means, that we rotate $S_{i_j}$ until no two points have the same $x$ coordinates, then apply the transformation 
        $(x,y_1,y_2,...y_{d-1})\rightarrow(\delta x,\epsilon y_1, \epsilon y_2,..., \epsilon y_{d-1})$ with $\delta$ tiny, and $\epsilon<<\delta$. Call the segment-like image $S_{i_j}'$. For clarity, check Figure \ref{fig:segment}:

 Now place a copy of each $S_{i_j}'$ with one of its points being at $\mathbf{p_j}$, and rotate it such that its $x$-vector aligns with the $\vec{\mathbf{pp_j}}$ vector. In other words, the tiny segment points towards the origin. Call that rotated and placed set $A_j$ for each $1\leq j\leq D$. We will also refer to these $D$ different sets in the construction as \textbf{blocks}. For a simple three-dimensional example, see Figure \ref{fig:construction}.

 Note that $g_d(S_{i_j})=g_d(S_{i_j}')=g_d(A_j)$ since flattening, scaling, rotating, and translating do not change the inclusion properties of the convex hulls in the point set.

 By the above process, we defined a set $S_n$ for all values of $n$. In the next section, we estimate the number of peeling sequences these constructions admit.

 \begin{figure}
	\begin{center}
		\includegraphics[scale=0.32]{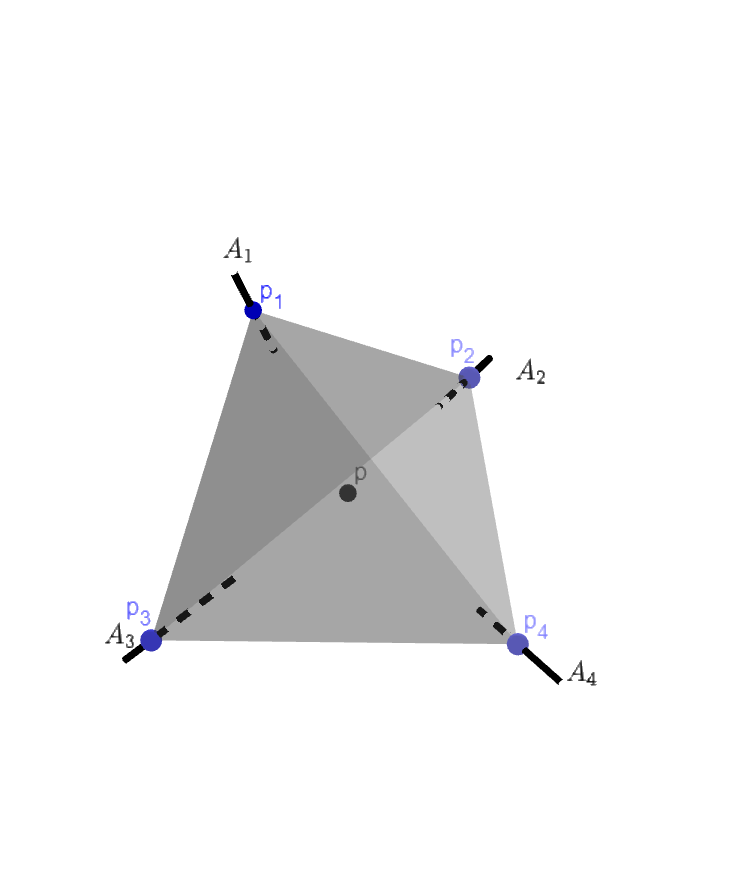}
		\caption{The recursive construction for $d=3$ and $m=1$.}
		\label{fig:construction}
	\end{center}
\end{figure}

\section{Proof of the upper bound}

 First, we need to define the concept of simplified peeling sequences, here we use the same definitions as Dumitrescu and Tóth in \cite{Dumitrescu}.

 \begin{definition*}[Simplified Peeling Sequence]
     Let $\pi$ be a peeling sequence of $S_n$. Replace every element of $A_j$ with a symbol $a_j$ in the sequence, for all $1\leq j\leq D$. The resulting sequence of $D$ different symbols is the simplified peeling sequence of $\pi$, and it is denoted $\pi^*$.
 \end{definition*}

We can think of the simplified peeling sequence as a way of seeing the order in which we remove the blocks of the construction. Each symbol $a_j$ just represents the block $A_j$.

\begin{lemma}\label{S_n}
    $g_d(S_n)\leq (D-m+1)^n\prod_{i=m+1}^{D}g_d(A_i)$, where we labelled the blocks such that $|A_1|\leq|A_2|\leq...\leq|A_D|$.
\end{lemma}

\begin{proof}
We can give an upper bound to the number of peeling sequences of $S_n$ the following way:
First, we bound from above the number of different simplified peeling sequences $S_n$ can admit. We also bound from above the maximum number of peeling sequences inducing the same simplified peeling sequence $\pi^*$. The product of these two values gives an upper bound for $g_d(S_n)$.
 
 Let $\pi_{A_j}$ denote the subsequence of $\pi$ containing all elements of $A_j$. Note that $\pi_{A_j}$ is a peeling sequence of $A_j$.

 Due to the construction of $S$, and the tiny segment-like nature of $A_j$, at any time we can only peel elements of at most $D-m+1$ different $A_j$-s. When all points in a block $A_j$ are removed, it introduces at most one additional block $A_i$ to the current boundary, ensuring that the number of blocks on the boundary remains at most $D-m+1$ at any given time.

 Hence the number of simplified peeling sequences of $S_n$ is bounded above by $(D-m+1)^n$.

 Now we need to bound the number of peeling sequences admitting the same simplified peeling sequence.

 As long as the origin is still defended by our point set during the peeling algorithm, only the outermost points (points with the largest distance from the origin) of $A_i$ may be removed for each $i$. This is because if the origin is defended, then a small ball is defended around it, hence the ball is in the current convex hull. Since $A_i$ is segment-like pointing towards the origin, any line between two points of $A_i$ intersects the ball at the center. If we take the outermost point $\mathbf{v}$ of $A_i$, then for all of the other points $\mathbf{v'}\in A_i$ the line $\mathbf{vv'}$ intersects the ball, hence $\mathbf{v'}$ must be in the interior of the current convex hull. Hence, only $\mathbf{v}$ can be removed in the current step if $A_i$ is on the boundary.

 Since $S$ defends the origin for $m$ steps in the peeling algorithm, in $S_n$ the origin remains defended as long as we have not removed $m$ distinct $A_i$ blocks entirely.

 So, consider the simplified peeling sequence $\pi^*$, and without loss of generality assume $A_1, A_2, ..., A_m$ are the first $m$ blocks that disappear entirely throughout the peeling. In other words, if we consider only the last $a_i$ for each $i$ in the simplified peeling sequence, $a_1, a_2, ..., a_m$ are the first $m$ elements.

 Then, for any $\pi$ admitting $\pi^*$ as its simplified peeling sequence, $\pi_{A_i}$  for $i\leq m$ are fixed since points of $A_i$ are removed one by one in decreasing order of distance from the origin.

 For $i\geq m+1$, we can use any upper bound we have for $g_d(A_i)$.

 The total number of peeling sequences admitting $\pi^*$ as the simplified peeling sequence is thus at most $\prod_{i=m+1}^{D}g_d(A_i)$. To get an upper bound we just relabel everything such that $g_d(A_1)\leq g_d(A_2) \leq ... \leq g_d(A_{D})$. This is equivalent with saying $|A_1|\leq|A_2|\leq...\leq|A_D|$. So for the upper bound we just assume that the remaining blocks are the largest possible out of all blocks, hence we have the maximum possible number of peeling sequences on the remaining blocks.

 Multiply that upper bound with the upper bound on the number of different simplified peeling sequences, and we proved the Lemma.
 \end{proof}

 Having settled the foundations, we can now prove Theorem \ref{0}. Let's recall the theorem for clarity.

%\begin{lemma}\label{mainlemma}
%For any fixed integers $d\geq3, m\geq1$ and all integers $n\geq2$,\\ $g_d(n)\leq c\cdot(D_d(m)-m+1)^{\frac{D_d(m)}{m}n}$ for a fixed universal constant $c$.
%\end{lemma}

\noindent \textbf{Theorem 2.} \textit{For any fixed integers $d\ge 3,\, m\geq1,$ and all integers $n\ge 2$, we have $$g_d(n)\leq c\cdot(d+m)^{\frac{d+2m-1}{m}n},$$ where $c$ is an absolute constant.}

\begin{proof}
Let $D=D_d(m)=d+2m-1$ as shown in Theorem \ref{maindefense}, and denote $a=(d+m)^{\frac{d+2m-1}{m}}$ for simplicity in the proof. Using the construction described in sections $3.1$ and $3.2$ for these fixed values of $d$ and $m$, we aim to show that the number of distinct peeling sequences of $S_n$ is at most $ca^n$. For the constant, we select $c=a^{-\frac{d}{d-1}}$.

 First, we establish the statement for $2\leq i\leq 2D$. 
 \begin{itemize}
     \item  For $2$ points we have $g_d(S_2)=2$, which is clearly less than $ca^2= a^{\frac{d-2}{d-1}}\geq4^{\frac{1}{2}}=2$ since $a\geq4$ and $d\geq3$.
     \item For $3\leq i\leq 2D$ points we use the upper bound $g_d(S_i)\leq i!$, which arises from placing all points in convex position. On the other hand, since $d+m\geq4$, we can show that $a\geq(d+m)^2\geq4(d+m)\geq 2d+4m-2=2D$.

     Hence $g_d(S_i)\leq i!\leq2(2D)^{i-2}\leq (ca^2)a^{i-2}=ca^i$ which we wanted to prove. Note that we used the $n=2$ case in the inequality.
    
 \end{itemize}

  We proceed by induction, assume $g_d(S_i)\leq ca^i$ for all $\lfloor\frac{n+1}{D}\rfloor<i<n$.

 Applying Lemma \ref{S_n}, and using that $g_d(A_i)\leq ca^{\frac{n}{D}+1}$ by the induction hypothesis, we derive: \newline
$g_d(S_n)\leq (D-m+1)^n\prod_{i=m+1}^{D}g_d(A_i) \leq (D-m+1)^n(ca)^{D-m}a^{\frac{D-m}{D}n}$.

 To complete the proof, we need to show that $(D-m+1)^n(ca)^{D-m}a^{\frac{D-m}{D}n}\leq ca^n$. We verify this separately for the coefficient and the exponential factors.
\begin{itemize}
    \item For the exponential factor we need $(D-m+1)^na^{\frac{D-m}{D}n}\leq a^n$, which holds if and only if $(D-m+1)^n\leq a^{\frac{m}{D}n}$. This is true, since $a=(D-m+1)^{\frac{D}{m}}=(d+m)^{\frac{d+2m-1}{m}}$.
    \item For the coefficient factor, we need to show $c\geq (ca)^{D-m}$. After rearranging, this requirement becomes $a^{\frac{m-D}{D-m-1}}\geq c$.  This inequality holds because: \\ $a^{\frac{m-D}{D-m-1}}=a^{-\frac{d+m-1}{d+m-2}}\geq a^{-\frac{d}{d-1}}=c$.
\end{itemize}

Thus, we conclude that \\
$g_d(S_n)\leq (D-m+1)^n\prod_{i=m+1}^{D}g_d(A_i) \leq (D-m+1)^n(ca)^{D-m}a^{\frac{D-m}{D}n}\leq ca^n$.

Hence by induction the theorem holds for any $n\geq2$.
\end{proof}

 To get the best possible upper bounds, we need to find the value $m$ that minimises $(d+m)^{\frac{d+2m-1}{m}}$. 
Using Theorem \ref{0}, we can easily deduce the following Corollary:

\begin{corollary}
    For any fixed $d\geq3$ and all $n\geq2$ integers,\\ $g_d(n)\leq c\cdot \min_{m}((d+m)^{\frac{d+2m-1}{m}n})$ for a positive real constant $c$. Minimum is taken over all positive integer values of $m$.
\end{corollary}

 Finding the minimum of that function generally is not straightforward, but for small values it can be optimised numerically, giving a much better upper bound than Theorem \ref{1} will. For general parameters, we can prove Theorem \ref{1}.

\begin{proof}[Proof of Theorem \ref{1}]
For general $d$, pick $m=\lfloor d\ln{d}\rfloor$ which is a decent approximation for the minimum point of the above expression, but not optimal in general. After substituting $m$ into Theorem \ref{0}, we get the upper bound \\
 $g_d(n)\leq c\cdot(d+d\ln{d})^{(2+\frac{d-1}{\lfloor d\ln{d}\rfloor})n}$, concluding the proof of Theorem \ref{1}.
\end{proof}

Corollary \ref{2} follows easily from here.

\begin{proof}[Proof of Corollary \ref{2}]
    Using Theorem \ref{1} and that $d$ is at least $3$,
    \\ $g_d(n)\leq O((d+d\ln{d})^{(2+\frac{d-1}{\lfloor d\ln{d}\rfloor})n})\leq O((2d\ln{d})^{(2+\frac{1}{\ln{d}})n})$. Now take logarithm to get \\ $(2+\frac{1}{\ln{d}})(\ln{d}+\ln{(2\ln{d})})= (2+\frac{1}{\ln{d}}+\frac{2\ln{(2\ln{d})}}{\ln{d}}+\frac{\ln{(2\ln{d})}}{(\ln{d})^2})\ln{d}\leq (2+\epsilon)\ln{d}$ whenever $\epsilon\geq \frac{1}{\ln{d}}+\frac{2\ln{(2\ln{d})}}{\ln{d}}+\frac{\ln{(2\ln{d})}}{(\ln{d})^2}$. Since the right side converges to $0$ as $d$ goes to infinity, for arbitrary small positive $\epsilon$ we can choose $d$ large enough. Hence the corollary holds.
\end{proof}

    \section{Concluding Remarks}

    The upper bounds provided in the paper are substantially closer to the best lower bounds than any previous result, but there is still room for improvement. Further analysis of these constructions, utilizing the tools from Section 3.2 in \cite{Dumitrescu}, could potentially yield stronger upper bounds. Nevertheless, even that analysis is not perfect, as refinement in the recursion can lead to better upper bounds at the expense of more intricate calculations. Unfortunately, we currently lack the tools to determine the exact number of peeling sequences in that specific construction. By developing some further counting methods it seems probable that one could improve the upper bound. Despite this, we do not anticipate the existence of much better constructions. It is possible that if we could accurately count the number of peeling sequences, the optimal value of $m$ for our construction might be $1$.
    
 \begin{problem}
     Determine the number of peeling sequences of the construction presented in this paper. Improve the upper bound on $g_d(n)$.
 \end{problem}
    
   The next logical step in the area should be improving the lower bound on the number of peeling sequences for any dimensions. While this question was initially posed in \cite{Dumitrescu}, I present it as a conjecture, emphasizing the need for further exploration in this direction. The existence of the fractional Erdős-Szekeres theorem \cite{Erdos-szekeres} seems to imply that there are many configurations when one can choose more than $d+1$ points to peel, making the lower bound appear weak, but we still could not improve the bound.

    \begin{conjecture}
        The minimal number of peeling sequences of $n$ points in $\R^d$ is at least $c(d+1+\epsilon)^n$ for some positive constants $c$ and $\epsilon$. So $g_d(n)=\Omega((d+1+\epsilon)^n)$.
    \end{conjecture}

    \section{Acknowledgements}
    I would like to thank János Pach for introducing me to Gale's theorem, and for his valuable assistance in improving the paper's writing style. I am also thankful for Géza Tóth's helpful comments.%optional

\end{document}